\documentclass[a4paper,12pt]{article}
\usepackage[text={6.5in,9.5in},centering]{geometry}
\usepackage{graphicx,url,subfig}
\RequirePackage[OT1]{fontenc}
\RequirePackage{amsthm,amsmath,amssymb,amscd}
\RequirePackage[numbers]{natbib}
\RequirePackage[colorlinks,citecolor=blue,urlcolor=blue]{hyperref}
\usepackage{enumerate}
\usepackage{datetime}
\usepackage{etex}
\usepackage[normalem]{ulem} 
\usepackage{soul} 
\makeatother
\numberwithin{equation}{section}
\allowdisplaybreaks

\theoremstyle{plain}
\newtheorem{theorem}{Theorem}[section]
\newtheorem{corollary}[theorem]{Corollary}
\newtheorem{lemma}[theorem]{Lemma}

\theoremstyle{definition}

\newtheorem{remark}[theorem]{Remark}

\newcommand{\E}{\mathbb{E}}
\newcommand{\W}{\dot{W}}
\newcommand{\ud}{\ensuremath{\mathrm{d}}}

\newcommand{\R}{\mathbb{R}}

\newcommand{\Erf}{\ensuremath{\mathrm{erf}}}

\newcommand{\Erfc}{\ensuremath{\mathrm{erfc}}}

\title{
  The third moment for the parabolic Anderson model}
\author{
{\bf Le Chen\footnote{University of Kansas, Department of Mathematics. 405 Snow Hall, 1460 Jayhawk Blvd. Lawrence, Kansas, 66045-7594, USA. {\it Email}: \url{chenle@ku.edu} or \url{chenle02@gmail.com}.}}
\\[1em]
University of Kansas
\date{\vspace{0em}\small \today}
}

\begin{document}
\maketitle
\begin{abstract}
In this paper, we study the {\it parabolic Anderson model} starting from the Dirac delta initial data:
\[
\left(\frac{\partial}{\partial t} -\frac{\nu}{2}\frac{\partial^2}{\partial x^2} \right)
u(t,x) = \lambda u(t,x) \W(t,x), \qquad u(0,x)=\delta_0(x), \quad x\in\R,
\]
where $\W$ denotes the space-time white noise.
By evaluating the threefold contour integral in the third moment formula by
Borodin and Corwin \cite{BC14Moment},
we obtain some explicit formulas for $\E[u(t,x)^3]$.
One application of these formulas is given to show the exact phase transition
for the intermittency front of order three.
\\

\noindent{\it Keywords.} The stochastic heat equation; 
	parabolic Anderson model;
	Dirac delta initial condition; space-time white noise; moment formula; intermittency fronts; growth indices.\\

	\noindent{\it \noindent AMS 2010 subject classification.}
	Primary 60H15; Secondary 35R60.
\end{abstract}

\setlength{\parindent}{1.5em}



\section{Introduction}

In this paper, we derive an explicit formula for
the third moment of the following {\it parabolic Anderson model} (PAM \cite{CarmonaMolchanov94}),
\begin{align}\label{E:SHE}
\begin{cases}
\displaystyle \left(\frac{\partial }{\partial t} -
\frac{\nu}{2}\frac{\partial^2}{\partial x^2} \right) u(t,x) =  \lambda u(t,x)
\:\dot{W}(t,x),&
x\in \R,\; t>0, \\
\displaystyle \quad u(0,\cdot) = \delta_0,
\end{cases}
\end{align}
where $\nu>0$ is the diffusion parameter and
$\delta_0$ is the Dirac delta measure with a unit mass at zero.
The noise is assumed to be the space-time white noise, i.e., 
\[
\E\left[\dot{W}(t,x)\dot{W}(s,y)\right] = \delta_0(t-s)\delta_0(x-y).
\]
The solution to \eqref{E:SHE} is understood in the mild form
\begin{align}\label{E:mild}
u(t,x)&= G_\nu(t,x) + \lambda \int_0^t \int_{\R^d} G(t-s,x-y)u(s,y)W(\ud s,\ud y),
\end{align}
where the stochastic integral in \eqref{E:mild} is in the sense of Walsh \cite{Walsh} and
\begin{align}
G_\nu(t,x)= \left(2\pi \nu t \right)^{-1/2} \exp\left(-\frac{x^2}{2\nu t}\right).
\end{align}

Explicit formulas both for the second moment and for the two-point correlation function were obtained in \cite{ChenDalang13Heat,ChenHuNualart15Two} for general initial measure $u(0,\cdot)=\mu(\cdot)$ such that $\left(|\mu|*G(t,\circ)\right)(x)<\infty$ for all $t>0$ and $x\in\R$. Here, $\mu=\mu_+-\mu_-$
is the {\it Jordan} decomposition and $|\mu|=\mu_++\mu_-$.
When the initial data is the delta measure with a unit mass at zero, i.e., $\mu=\delta_0$, these formulas reduce to the following simple forms (see Corollary 2.8 of \cite{ChenDalang13Heat}):
\begin{align}
\label{E2:TP-Delta}
\begin{aligned}
\E\left[u(t,x)u\left(t,y\right)\right] =
G_\nu(t,x)G_\nu\left(t,y\right)
+ &\frac{\lambda^2}{4\nu}G_{\nu/2}\left(t,\frac{x+y}{2}\right)\\
\times &
\exp\left(\frac{\lambda^4 t-2 \lambda^2 |x-y|}{4 \nu }\right)
   \Erfc\left(\frac{|x-y|-\lambda^2 t}{2
   \sqrt{\nu  t}}\right),
   \end{aligned}
\end{align}
and in particular, 
\begin{align}
\E\left[u(t,x)^2\right]=
G_{\nu/2}(t,x)
\left(\frac{\lambda^2}{\sqrt{4\pi\nu t}}+\frac{\lambda^4}{2\nu} e^{\frac{\lambda^4t}{4\nu}}
\Phi\left(\lambda^2\sqrt{\frac{t}{2\nu}}\right)\right).
\end{align}
In these formulas, we have used the following special functions
\[
\Phi(x) = \int_{-\infty}^x (2\pi)^{-1/2}e^{-y^2/2} \ud y, \quad
\Erf(x)=\frac{2}{\sqrt{\pi}}\int_0^x e^{-y^2}\ud y,\quad
\Erfc(x)=1-\Erf(x).
\]
Bertini and Cancrini obtained
the two-point correction function in an integral form; see Corollary 2.5 in \cite{BertiniCancrini94Intermittence}. This 
integral has been evaluated explicitly in \cite{ChenDalang13Heat},
which gives the same form as \eqref{E2:TP-Delta}.

\bigskip
On the other hand, for the delta initial condition, in a beautiful work by Borodin and Corwin \cite{BC14Moment}, it is showed that
for any $x_1\le \dots\le x_k$, 
\begin{align}\label{E:BCMac}
\E\left[\prod_{j=1}^k u(t,x_j)\right] 
=\frac{1}{(2\pi \imath )^k}\int\cdots\int \prod_{1\le A < B\le k}\frac{z_A-z_B}{z_A-z_B-\frac{\lambda^2}{\nu}}
\prod_{j=1}^k \exp\left(\frac{\nu t}{2} z_j^2 +x_jz_j\right)\ud z_j,
\end{align}
where $\imath=\sqrt{-1}$ and the $z_j$ integration is over $\alpha_j+\imath\R$ with
\begin{align}\label{E:alpha}
\alpha_1>\alpha_2+\frac{\lambda^2}{\nu}>\alpha_3+\frac{2\lambda^2}{\nu}>\cdots;
\end{align}
see Appendix A.2 in \cite{BC14Moment} where they assume $\nu=1$.
Thanks to condition \eqref{E:alpha}, this formula \eqref{E:BCMac} can be transformed into the following form by introducing another $k(k-1)/2$ integrals, which results in a formula with  $k(k+1)/2$ integrals:
\begin{align}\label{E:BCMac2}
\begin{aligned}
\E\left[\prod_{j=1}^k u(t,x_j)\right] 
=&\frac{1}{(2\pi \imath )^k}\int\cdots\int \quad \prod_{j=1}^k \exp\left(\frac{\nu t}{2} z_j^2 +x_jz_j\right)\ud z_j
\\
&\times \prod_{1\le A < B\le k}(z_A-z_B)
\int_0^\infty \exp\left(-s_{AB}\left[z_A-z_B-\frac{\lambda^2}{\nu}\right]\right)\ud s_{AB}.
\end{aligned}
\end{align}

\bigskip

In the following, we will first show that when $k=2$, one can recover \eqref{E2:TP-Delta}
by evaluating the double contour integrals in \eqref{E:BCMac}. 
Then we proceed to derive some formulas, more explicit than \eqref{E:BCMac}, for the third moment.
As an application, we establish the third exact phase transition (see below).
These results are summarized in the following three Theorems \ref{T:2ndMoment}, \ref{T:3rdMoment} and \ref{T:Front3rd}.

\begin{theorem}[Second moment]
\label{T:2ndMoment}
When $k=2$, the double contour integrals on the right-hand side of \eqref{E:BCMac} are equal to 
\[
G_{\nu}(t,x_1)G_{\nu}(t,x_2)+\frac{\lambda^2}{4\nu}
 G_{\nu/2}\left(t,\frac{x_1+x_2}{2}\right)
 \exp\left(\frac{t\lambda^4-2\lambda^2(x_2-x_1)}{4\nu}\right)
 \Erfc\left(\frac{(x_2-x_1)-t\lambda^2}{2\sqrt{\nu t}}\right),
\]
which recovers \eqref{E2:TP-Delta} (with $x_1\le x_2$).
\end{theorem}

Note that in case of $k=3$, there are six integrals in \eqref{E:BCMac2}.
In the next theorem, we will first evaluate the three contour integrals in \eqref{E:BCMac2} which leads to \eqref{E:3point}.
Then we proceed to evaluate two real integrals in \eqref{E:3point} leading to \eqref{E:3moment}.
Finally, by applying the mean-value theorem, we evaluate
the last real integral to obtain an explicit expression which is handy for applications; see \eqref{E:3momentAp}. 

\begin{theorem}[Third moment]
\label{T:3rdMoment}
Suppose that the initial data is the Dirac delta measure $\delta_0(x)$. The following statements are true:\\
(1) For all $t>0$ and $x_i\in\R$, $i=1,2,3$,
\begin{align}\notag
 &\E\left[u(t,x_1)u(t,x_2)u(t,x_3)\right]\\
 \notag
= & (\nu t)^{-9/2} (2\pi)^{-3/2}
\int_0^\infty \ud s_1 \: 
\exp\left(\frac{\lambda^2 s_1}{\nu}\right)
\int_0^\infty \ud s_2 \: 
\exp\left(\frac{\lambda^2 s_2}{\nu}\right)
\int_0^\infty \ud s_3 \: 
\exp\left(\frac{\lambda^2 s_3}{\nu}\right)\\
\notag
&\times
\left(-s_1+s_2+2 s_3+|x_3-x_2|\right)
\left( s_1+2 s_2+s_3+|x_3-x_1|\right)
\left( 2 s_1+s_2-s_3+|x_2-x_1|\right) 
\\
&\times 
\exp\left(
-\frac{(x_3+s_2+s_3)^2}{2\nu t}
-\frac{(x_2+s_1-s_3)^2}{2\nu t}
-\frac{(x_1-s_1-s_2)^2}{2\nu t}\right).
\label{E:3point}
\end{align}
(2) For all $x\in\R$, 
\begin{align}
\notag
\E[u(t,x)^3] 
=&  \frac{ \lambda^2}{2^{5/2}\pi\:\nu^3t^2}\exp\left(\frac{\lambda^2 t}{4\nu}-\frac{3x^2}{2\nu t}\right)\\
\notag
&\times 
\int_0^\infty  
 \Bigg[
 \left(3
   s+\lambda^2 t\right) 
   \exp\left(
   -\frac{3 s(s-2 \lambda^2 t)}{4 \nu t}
   \right)
   \Phi\left(\frac{\lambda^2 t+s}{\sqrt{2\nu t}}\right)
   \\
   \notag
 &\hspace{3.2em}+\left(3 s-\lambda^2 t\right)
      \exp\left(-\frac{ s(3s-2 \lambda^2 t)}{4 \nu t}\right)
      \Phi\left(\frac{\lambda^2 t-s}{\sqrt{2\nu t}}
      \right)\Bigg]\ud s\\
      \notag
 &+\frac{1}{(2\pi\nu t)^{3/2}}
 \exp\left(-\frac{3x^2}{2\nu t}\right)\\
 &  +\frac{\lambda^2}{2\sqrt{2} \pi\nu^2 t}
   \exp\left(\frac{\lambda^4 t}{4\nu}-\frac{3x^2}{2\nu t }\right) \left(1-\Phi\left(-\lambda^2 \sqrt{\frac{t}{2\nu}}\right)\right).
\label{E:3moment}
\end{align}
(3) For all $t>0$ and $x\in\R$, there are some two constants $a,b$ depending on $t$, $\nu$ and $\lambda$ in the following range
\begin{align}\label{E:range-ab}
0\le a\le \Phi\left(\lambda^2\sqrt{\frac{t}{2\nu}}\right) \le b
\le 1
\end{align}
such that 
\begin{subequations}
\label{E:3momentAp}
\begin{align}
 \E\left[u(t,x)^3\right] 
 =&\quad
\frac{1}{(2\pi\nu t)^{3/2}}
\exp\left(-\frac{3 x^2}{2 \nu t}\right)\\
   &+
   \frac{\lambda^2 (b-a)}{2 \sqrt{2}
   \pi  \nu^2 t}\exp\left(\frac{\lambda^4 t}{4 \nu}-\frac{3x^2}{2 \nu t}\right)\\
&+
\frac{\lambda^2}{2\sqrt{2} \pi 
   \nu ^2 t}\exp\left(\frac{\lambda^4 t}{4\nu
   }-\frac{3 x^2}{2 \nu  t}\right)
   \Phi\left(
   \lambda^2
   \sqrt{\frac{t}{2\nu
   }}\right)   \label{E:3momentC}\\
&+\frac{b \sqrt{2}\:\lambda^4}{\sqrt{3 \pi \nu^5 t}}
\exp\left(\frac{\lambda^4 t}{\nu
   }-\frac{3 x^2}{2 \nu  t}\right)
   \Phi\left(\lambda^2
   \sqrt{\frac{3t}{2\nu}}\right).
    \label{E:3momentD}
\end{align}
\end{subequations}
\end{theorem}

\begin{remark}
It is known that $u(t,x)$ is strictly positive for all
$t>0$ and $x\in\R$ a.s.; see \cite{ChenKim14Comp}. Hence,
$\E\left[|u(t,x)|^3\right]=\E\left[u(t,x)^3\right]$.
\end{remark}

In the following,  let $u_\delta$ and $u_1$ denote the solutions to \eqref{E:SHE}
starting from the delta measure $\delta_0$ and Lebesgue's measure (i.e., $u(0,x)=1$), respectively.
The following corollary is a consequence of \eqref{E:3point} and the fact that 
\begin{align}
 \E\left[u_1(t,x)^k\right] = \int_{\R}\cdots \int_{\R}
 \E\left[\prod_{i=1}^k u_\delta(t,x_i)\right] \ud x_1\dots \ud x_k.
\end{align}

\begin{corollary}
 Suppose that the initial data is Lebesgue's measure, i.e., $u(0,x) = 1$. Then 
 \begin{align}\notag
\E\left[u_1(t,x)^3\right]
= & (\nu t)^{-9/2} (2\pi)^{-3/2}
\int_\R\ud x_1\int_\R\ud x_2\int_\R\ud x_3
\\ 
\notag
&\times \int_0^\infty \ud s_1 \: 
\exp\left(\frac{\lambda^2 s_1}{\nu}\right)
\int_0^\infty \ud s_2 \: 
\exp\left(\frac{\lambda^2 s_2}{\nu}\right)
\int_0^\infty \ud s_3 \: 
\exp\left(\frac{\lambda^2 s_3}{\nu}\right)\\
\notag
&\times
\left(-s_1+s_2+2 s_3+|x_3-x_2|\right)\\
\notag
&\times
\left( \quad s_1+2 s_2+s_3+|x_3-x_1|\right)\\
\notag
&\times \left( \quad  2 s_1+s_2-s_3+|x_2-x_1|\right) 
\\
&\times 
\exp\left(
-\frac{(x_3+s_2+s_3)^2}{2\nu t}
-\frac{(x_2+s_1-s_3)^2}{2\nu t}
-\frac{(x_1-s_1-s_2)^2}{2\nu t}\right).
\label{E:3m-Leb}
\end{align}
\end{corollary}

\begin{remark}
\label{R:kmoment}
Bertini and Cancrini studied \eqref{E:SHE} with $\lambda=1$ and claimed in Theorem 2.6 of \cite{BertiniCancrini94Intermittence} 
that 
\begin{align}\label{E:kmoment}
 \E\left[u_1(t,x)^k\right]
 &=2 \exp\left(\frac{k(k^2-1)}{4!\: \nu} t\right)
 \Phi\left(\sqrt{\frac{k(k^2-1)}{12\nu}t}\:\right).
\end{align}
X. Chen showed in \cite{ChenX15} that 
\eqref{E:kmoment} is correct only for $k=2$; see also Remark 2.6 in \cite{ChenDalang13Heat}.
Note that there are six integrals in \eqref{E:3m-Leb}. After integrals over $\ud x_2\ud x_3$, the expression becomes too complicated and too long to handle. 
We leave it to interested readers to simplify this formula.
\end{remark}

\begin{remark}[Asymptotics]
\label{R:asymp}
Note that the leading orders for large time $t$ both in \eqref{E:kmoment} with $k=3$ for $u_1$ and in \eqref{E:3momentAp} for $u_\delta$ are the same.
Actually, X. Chen \cite{ChenX15} showed that the asymptotics of the $k$-th moment
of $u_1$ (note that it is not $u_\delta$) \footnote{Note that in \cite{ChenX15}, the initial data is assumed to a nonnegative function that satisfies 
\[
0<\inf_{x\in\R} u(0,x)\le \sup_{x\in\R} u(0,x)<\infty.
\]
For studying asymptotic properties, this assumption is essentially equivalent to the case that $u(0,x)\equiv 1$.} does satisfy, as \eqref{E:kmoment}, that
\begin{align}
\lim_{t\rightarrow\infty}\frac{1}{t}\log\E\left[u_1(t,x)^k\right] =\frac{k(k^2-1)}{4! \:\nu}.
\end{align}
As an easy consequence of \eqref{E:3momentAp}, we see that this fact is still true for $u_\delta$, namely,
\begin{align}
\lim_{t\rightarrow\infty}\frac{1}{t}\log\E\left[u_\delta(t,x)^3\right] =
\left.\frac{\lambda^4 k(k^2-1)}{4! \:\nu}\right|_{k=3}
=\frac{\lambda^4}{\nu},\qquad\text{for all $x\in\R$.} 
\end{align}
Of course, the formula \eqref{E:3momentAp} itself contains more information
than these asymptotics.
\end{remark}
\bigskip

Now we state one application of these moment formulas.
It is known that the solution to \eqref{E:SHE} with a general initial condition is intermittent \cite{CarmonaMolchanov94, ChenDalang13Heat, FK09EJP}, which means informally that the solution in question develops many tall peaks.
An interesting phenomenon is that when the initial data has compact support, these tall peaks will propagate in space with certain speed depending on the value of $\lambda$; see some simulations in Figure \ref{fig}.
The spatial fronts of these tall peaks are called the {\it intermittency fronts}.
Conus and Khoshnevisan \cite{ConusKhosh10Farthest}
introduced the following {\it lower and upper growth indices of order $p$} to characterize these intermittency fronts,
\begin{align}
\label{E1:GrowInd-0}
\underline{\lambda}(p):= &
\sup\left\{\alpha>0: \underset{t\rightarrow \infty}{\lim\sup}
\frac{1}{t}\sup_{|x|\ge \alpha t} \log \E\left(|u(t,x)|^p\right) >0
\right\}\;,\\
\label{E1:GrowInd-1}
 \overline{\lambda}(p) := &
\inf\left\{\alpha>0: \underset{t\rightarrow \infty}{\lim\sup}
\frac{1}{t}\sup_{|x|\ge \alpha t} \log \E\left(|u(t,x)|^p\right) <0
\right\}\:.
\end{align}
We call the case $\underline{\lambda}(p) = \overline{\lambda}(p)$
the {\em $p$-th exact phase transition}.
Chen and Dalang \cite{ChenDalang13Heat} established the second exact phase
transition, namely, if the initial data is a nonnegative measure with compact
support, then
\begin{align}\label{E:Front2nd}
 \lambda(2):= \underline{\lambda}(2) =\overline{\lambda}(2)
=\frac{1}{2}\: \lambda^2.
\end{align}
This improves the result by Conus and Khoshnevisan \cite{ConusKhosh10Farthest}
that $(2\pi)^{-1}\lambda^2
\le \underline{\lambda}(2) \le \overline{\lambda}(2)
\le 2^{-1}\lambda^2$.
Chen and Kunwoo \cite{CK15SHE} studied the stochastic heat equation --
the equation with $\lambda u$ in \eqref{E:SHE}  replaced by $\sigma(u)$ --
on $\R^d$ subject to a Gaussian noise that is white in time and colored in
space. Both nontrivial lower and upper bounds for the second growth indices were
obtained. More recently,  Huang, L\^e and Nualart studied the PAM on $\R^d$ with
a Gaussian noise that may have
colors in both space and time; see \cite{HLN15} and \cite{HLN16}. They
obtained some nontrivial, sometimes
matching, bounds for the lower and upper growth indices of all orders $p\ge 2$.
In particular, they showed in the current setting that 
\begin{align}\label{E:HLN}
\lambda(p) = \underline{\lambda}(p)=\overline{\lambda}(p) =
\sqrt{\frac{p^2-1}{12}} \lambda^2.
\end{align}
The following theorem is an easy corollary of our moment formula, which
recovers their result \eqref{E:HLN} for $p=3$.

\begin{figure}[h!tbp]
 \centering
  \subfloat[$\lambda=3$]{\includegraphics[scale=0.4]{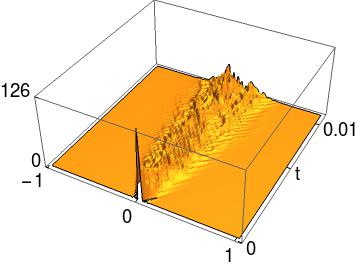}}%
 \hfill
 \subfloat[$\lambda=4$]{\includegraphics[scale=0.4]{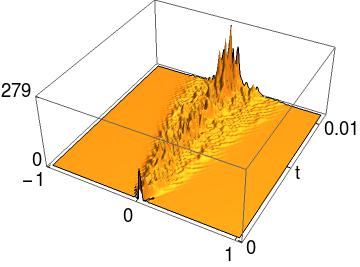}}%
 \hfill
 \subfloat[$\lambda=5$]{\includegraphics[scale=0.4]{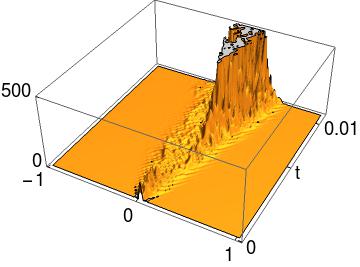}}
  \hfill
 \subfloat[$\lambda=6$]{\includegraphics[scale=0.4]{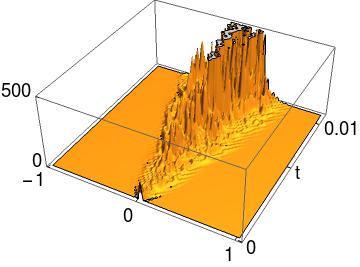}}
 \hfill
 \subfloat[$\lambda=7$]{\includegraphics[scale=0.4]{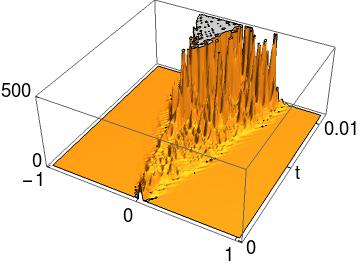}}
 \hfill
 \subfloat[$\lambda=8$]{\includegraphics[scale=0.4]{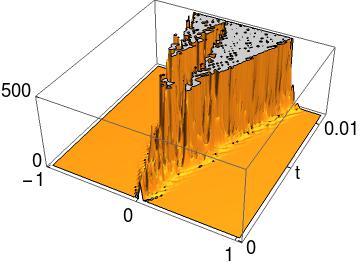}}
 \hfill
 \caption{Some simulations of the solution to \eqref{E:SHE} with various
$\lambda$ ranging from $3$ to $8$. These simulations are truncated at the level
$500$ in order to compare the size of these six space-time cones. The initial
data is $(2\pi \mu)^{-1/2}\exp\left(-x^2/(2\mu)\right)$ with $\mu=10^{-5}$,
which is an approximation to the delta initial measure.
 These simulations are made up to $t=0.01$.}
 \label{fig}
\end{figure}

\begin{theorem}[The third exact phase transition]
\label{T:Front3rd}
If $u(0,x)=\delta_0(x)$, then 
\begin{align}\label{E:Front3rd}
 \lambda(3):=\overline{\lambda}(3)=\underline{\lambda}(3)= 
\sqrt{\frac{2}{3}}\: \lambda^2.
\end{align}
\end{theorem}
\begin{proof}
Because $b$ in \eqref{E:range-ab} goes to $1$ as $t$ goes to infinity, we see that
\begin{align}
\label{E_:limsup}
\lim_{t\rightarrow\infty}\frac{1}{t}\sup_{|x|\ge \alpha t}\log\E\left(u(t,x)^p\right) = \frac{\lambda^4}{\nu}-\frac{3\alpha^2}{2\nu}.
\end{align}
Hence, 
\begin{align}
\lim_{t\rightarrow\infty}\frac{1}{t}\sup_{|x|\ge \alpha t}\log\E\left(u(t,x)^p\right)> 0
\qquad\Longleftrightarrow\qquad
\alpha < \sqrt{\frac{2}{3}} \:\lambda^2, 
\end{align}
which shows that $\underline{\lambda}(3)=\sqrt{2/3} \:\lambda^2$. Similarly, one can show that 
$\overline{\lambda}(3)=\sqrt{2/3} \:\lambda^2$.
\end{proof}

In the following, we will prove Theorems \ref{T:2ndMoment} and \ref{T:3rdMoment} in Sections \ref{S:2ndMoment} and \ref{S:3rdMoment}, respectively.

\section{Second moment: Proof of Theorem \ref{T:2ndMoment}}\label{S:2ndMoment}
The proof of Theorem \ref{T:2ndMoment} severs as a warm-up for the more involved proof of Theorem \ref{T:3rdMoment}.

\begin{proof}[Proof of Theorem \ref{T:2ndMoment}]
Suppose that $x_1\le x_2$. Fix any $\alpha_1>\alpha_2+\frac{\lambda^2}{\nu}$. Then,
from \eqref{E:BCMac2},
\begin{align}
\notag
\E\left[u(t,x_1)u(t,x_2)\right]=&\frac{1}{(2\pi\imath)^2}\iint
\left(\int_0^\infty \exp\left(-\left[ (z_1-z_2) - \frac{\lambda^2}{\nu}\right]s\right)\ud s \right)\\ \notag
&\times (z_1-z_2)\exp\left(\frac{\nu t}{2}z_1^2+x_1z_1\right)
\exp\left(\frac{\nu t}{2}z_2^2+ x_2z_2\right)\ud z_1\ud z_2\\ \notag
=& 
\frac{1}{(2\pi\imath)^2}
\int_0^\infty \ud s \exp\left(\frac{\lambda^2}{\nu}s\right)
\int \ud z_1 \exp\left(\frac{\nu t}{2}z_1^2+(x_1-s)z_1\right)\\
&\times\int \ud z_2 \:
(z_1-z_2)\exp\left(\frac{\nu t}{2}z_2^2+(x_2+s)z_2\right).
\end{align}
The $\ud z_2$ integration over $\alpha_2+\imath\R$ is equal to
\begin{align}\notag
 \int \ud z_2 \:&
(z_1-z_2)\exp\left(\frac{\nu t}{2}z_2^2+(x_2+s)z_2\right)\\ \notag
&=
\frac{\sqrt{2\pi}}{\sqrt{\nu t}} \exp\left(-\frac{(x_2+s)^2}{2\nu t}\right)\\ \notag
&\quad\times \int \ud z_2 \:
\left(\left[z_1+\frac{x_2+s}{\nu t}\right]-
\left[z_2+\frac{x_2+s}{\nu t}\right]\right)
\frac{\exp\left(\frac{\nu t}{2}\left[z_2+\frac{x_2+s}{\nu t} \right]^2\right)}{\sqrt{2\pi/(\nu t)}}
\\
&=\imath \frac{\sqrt{2\pi}}{(\nu t)^{3/2}}
\exp\left(-\frac{\left(s+x_2\right)^2}{2 \nu 
   t}\right) \left( \nu  t z_1+ (x_2+s)\right).
\end{align}
The $\ud z_1$ integration over $\alpha_1+\imath\R$ gives
\begin{align}\notag
\imath \int  &\exp\left(\frac{\nu t}{2}z_1^2+ (x_1-s)z_1\right) 
\frac{\sqrt{2 \pi }}{(\nu t)^{3/2}}
\exp\left(-\frac{\left(s+x_2\right)^2}{2 \nu 
   t}\right) \left( \nu  t z_1+(x_2+s)\right)\ud z_1\\
   \notag
=&
\imath\frac{2\pi}{\nu t} 
\exp\left(-\frac{(x_1-s)^2+(x_2+s)^2}{2 \nu  t}\right)\\
\notag
&\times \int \ud z_1 
\frac{1}{\sqrt{2\pi/(\nu t)}}\exp\left(\frac{\nu t}{2}\left[z_1+ \frac{x_1-s}{\nu t}\right]^2\right) \left(\left[z_1+\frac{x_1-s}{\nu t}\right]+\frac{x_2-x_1+2s}{\nu t}\right)
\\
=&
 \imath^2 \frac{ 2  \pi  }{(\nu t)^2}\exp\left(-\frac{(x_1-s)^2+(x_2+s)^2}{2 \nu  t}\right)\left(2 s+x_2-x_1\right).
\end{align}
Hence,
\begin{align}\notag
 \E&[u(t,x_1)u(t,x_2)]\\ \notag
 &=\int_0^\infty
 \frac{ \left(2 s+x_2-x_1\right)}{2\pi(\nu t)^2}\exp\left(-\frac{(x_1-s)^2+(x_2+s)^2}{2 \nu  t} +
 \frac{\lambda^2 s}{\nu}\right)\ud s\\ \notag
 &=
 \frac{1}{\pi\nu t}\exp\left(-\frac{x_1^2+x_2^2}{2\nu t}+\frac{(x_1-x_2+t\lambda^2)^2}{4\nu t}\right)
 \\ \notag &\qquad \times 
 \int_0^\infty
 \frac{ \left(s- \frac{x_1-x_2}{2}\right)}{\nu t}
 \exp\left(-\frac{1}{\nu  t}\left(s-\frac{x_1-x_2+t\lambda^2}{2}\right)^2\:\right)\ud s\\ \notag
 &=
 \frac{1}{\pi\nu t}\exp\left(-\frac{x_1^2+x_2^2}{2\nu t}+\frac{(x_1-x_2+t\lambda^2)^2}{4\nu t}\right)
 \int_{\frac{x_2-x_1-t\lambda^2}{2\sqrt{\nu t}}}^\infty
 \left(z+\frac{\lambda^2\sqrt{t}}{\sqrt{\nu}}\right)
 e^{-z^2}\ud z\\ \notag
 &=
 \frac{1}{\pi\nu t}\exp\left(-\frac{x_1^2+x_2^2}{2\nu t}+\frac{(x_1-x_2+t\lambda^2)^2}{4\nu t}\right)
 \\ \notag
&\qquad\times \left(\frac{1}{2}\exp\left(-\frac{(x_1-x_2+t\lambda^2)^2}{4\nu t}\right)
 +\frac{\lambda^2 \sqrt{\pi t}\: }{2\sqrt{\nu}}\Erfc\left(\frac{x_2-x_1-t\lambda^2}{2\sqrt{\nu t}}\right)\right)\\ \notag
 &=G_{\nu}(t,x_1)G_{\nu}(t,x_2)\\ 
 &\qquad+\frac{\lambda^2}{4\nu}
 G_{\nu/2}\left(t,\frac{x_1+x_2}{2}\right)
 \exp\left(\frac{t\lambda^4-2\lambda^2(x_2-x_1)}{4\nu}\right)
 \Erfc\left(\frac{(x_2-x_1)-t\lambda^2}{2\sqrt{\nu t}}\right).
\end{align}
This proves Theorem \ref{T:2ndMoment}.
\end{proof}

\section{Third moment: Proof of Theorem \ref{T:3rdMoment}}
\label{S:3rdMoment}

We first prove two lemmas. 
For all $\beta\in\R$, $t>0$ and $n\ge 0$, denote 
\begin{align}\label{E:DefLambda}
\Lambda_n(\beta,t):=\int_0^\infty \exp\left(-\frac{s^2}{t}+\beta s\right)s^n
\ud s.
\end{align}
\begin{lemma}
\label{L:RealInt}
For all $t>0$ and $\beta\in\R$, 
\begin{align}\label{E:L0}
\Lambda_0(\beta,t) &=
\frac{\sqrt{\pi t}}{2} \: \exp\left(\frac{\beta^2
   t}{4}\right) \left(2-\Erfc\left(\frac{\beta 
   \sqrt{t}}{2}\right)\right),\\
    \label{E:L1}
\Lambda_1(\beta,t) &=
\frac{t}{2}+\frac{\sqrt{\pi}  t^{3/2}\beta }{4}
   \exp\left(\frac{\beta ^2 t}{4}\right)
   \left(2-\Erfc\left(\frac{\beta 
   \sqrt{t}}{2}\right)\right),
\end{align}
and for all $n\ge 2$,
\begin{align}\label{E:Lambda}
\Lambda_n(\beta,t)  = \frac{\beta t}{2}\Lambda_{n-1}(\beta,t)+ \frac{ t (n-1)}{2}\Lambda_{n-2}(\beta,t).
\end{align}
In particular, we have that 
\begin{align}
   \label{E:L2}
\Lambda_2(\beta,t)&=\frac{\beta t^2}{4} +\frac{\sqrt{\pi } t^{3/2} \left(\beta ^2 t+2\right)}{8}
\exp\left(\frac{\beta^2t}{4}\right) 
   \left(2-\Erfc\left(\frac{\beta 
   \sqrt{t}}{2}\right)\right),\\
\label{E:L3}
\Lambda_3(\beta,t)&=\frac{t^2}{2}+\frac{\beta^2 t^3}{8} + 
\frac{\sqrt{\pi}\beta t^{5/2}\left(\beta ^2 t+6\right)}{16}
\exp\left(\frac{\beta^2 t}{4}\right)
\left(2-\Erfc\left(\frac{\beta 
   \sqrt{t}}{2}\right)\right).
\end{align}
\end{lemma}
\begin{proof}
We first calculate $\Lambda_0(\beta,t)$:
\begin{align*}
\Lambda_0(\beta,t) &=
\int_0^\infty \exp\left(-\frac{1}{t}\left(s-\frac{t\beta}{2}\right)^2+\frac{t\beta^2}{4}\right)
\ud s\\
&=\sqrt{t}\: \exp\left(\frac{t\beta^2}{4}\right)
\int_0^\infty \exp\left(-\left(z-\frac{\beta\sqrt{t}}{2}\right)^2\right)\ud z\\
&= 
\sqrt{t}\: \exp\left(\frac{t\beta^2}{4}\right)
\frac{\sqrt{\pi}}{2} \left(1+\Erf\left(\frac{\beta\sqrt{t}}{2}\right)\right).
\end{align*}
As for $\Lambda_1(\beta,t)$, 
\begin{align*}
\Lambda_1(\beta,t)&= -\frac{t}{2}\int_0^\infty \left[\left(-\frac{s^2}{t}+\beta s\right)'-\beta\right] \exp\left(-\frac{s^2}{t}+\beta s\right)  \ud s\\
&=-\frac{t}{2} \int_0^\infty  \ud \exp\left(-\frac{s^2}{t}+\beta s\right)  + 
\frac{\beta t}{2} \Lambda_0(\beta,t)\\
&= 
\frac{t}{2} + 
\frac{\beta t}{2} \Lambda_0(\beta,t),
\end{align*}
which is equal to the right-hand side of \eqref{E:L1} after simplification. 
For any $n\ge 2$, 
\begin{align*}
\Lambda_n(\beta,t)&= -\frac{t}{2}\int_0^\infty \left[\left(-\frac{s^2}{t}+\beta s\right)'-\beta\right]s^{n-1} \exp\left(-\frac{s^2}{t}+\beta s\right)  \ud s\\
&=-\frac{t}{2} \int_0^\infty s^{n-1} \ud \exp\left(-\frac{s^2}{t}+\beta s\right)  + 
\frac{\beta t}{2} \Lambda_{n-1}(\beta,t)\\
&= 
-\frac{t}{2} \left( 0- (n-1)\int_0^\infty s^{n-2} \exp\left(-\frac{s^2}{t}+\beta s\right)   \ud s \right)+ 
\frac{\beta t}{2} \Lambda_{n-1}(\beta,t)\\
&= \frac{t(n-1)}{2} \Lambda_{n-2}(\beta,t) + \frac{\beta t}{2} \Lambda_{n-1}(\beta,t),
\end{align*}
which proves \eqref{E:Lambda}.
In particular,
\begin{align*}
 \Lambda_2(\beta,t) &=
\frac{\beta t^2}{4} + \left(\frac{t}{2}+\frac{\beta^2 t^2}{4}\right)\Lambda_0(\beta,t),\\
\Lambda_3(\beta,t)&=
\frac{t^2}{2}+\frac{\beta^2 t^3}{8} + \left(\frac{3 \beta  t^2}{4}+\frac{\beta ^3 t^3}{8}\right)\Lambda_0(\beta,t).
\end{align*}
Then by \eqref{E:L0}, one obtains \eqref{E:L2} and \eqref{E:L3}.
This completes the proof of Lemma \ref{L:RealInt}.
\end{proof}

The next lemma is a corollary of Lemma \ref{L:RealInt}. We nevertheless state it separately for the convenience of our application.

\begin{lemma}\label{L:asbcsd}
For all $a,b,d\in\R$ and $c>0$, it holds that
\begin{align}
 \int_0^\infty (as+b)\exp\left(-cs(s-d)\right)\ud s
 =\frac{\sqrt{\pi } (a d+2b)}{2\sqrt{c}}
   \exp\left(\frac{c d^2}{4}\right) \Phi\left(\frac
   {\sqrt{2c}\:d}{2}\right)+\frac{a}{2c}.
\end{align}
\end{lemma}

\bigskip
Now we are ready to prove Theorem \ref{T:3rdMoment}.

\begin{proof}[Proof of Theorem \ref{T:3rdMoment}]
{\bf\noindent (1)~}
Fix $x_1, x_2, x_3\in\R$ with $x_1\le x_2\le x_3$, $t>0$, and $\alpha_1,\alpha_2,\alpha_3\in\R$ with $\alpha_1>\alpha_2+\frac{\lambda^2}{\nu}>\alpha_3+2\frac{\lambda^2}{\nu}$. 
Similar to the proof of Theorem \ref{T:2ndMoment},
from \eqref{E:BCMac2}, we have that
\begin{align}\notag
\E&\left[u(t,x_1)u(t,x_2)u(t,x_3)\right]\\ \notag
&= 
\frac{1}{(2\pi\imath)^3}
\int_0^\infty \ud s_1 \: \exp\left(\frac{\lambda^2}{\nu}s_1\right)
\int_0^\infty \ud s_2 \: \exp\left(\frac{\lambda^2}{\nu}s_2\right)
\int_0^\infty \ud s_3 \: \exp\left(\frac{\lambda^2}{\nu}s_3\right)\\ \notag
&\quad\times 
\int\ud z_1
\exp\left(\frac{\nu t}{2}z_1^2+(x_1-s_1-s_2)z_1\right)\\ \notag
&\quad\times
\int\ud z_2
\exp\left(\frac{\nu t}{2}z_2^2+(x_2+s_1-s_3)z_2\right)\\ 
&\quad\times 
\int \ud z_3
(z_1-z_2)(z_1-z_3)(z_2-z_3)
\exp\left(\frac{\nu t}{2}z_3^2+(x_3+s_2+s_3)z_3\right).
\end{align}
Now we will calculate the triple integral over $\ud z_3\ud z_2\ud z_1$.
Recall that $x_1\le x_2\le x_3$.
The $\ud z_3$-integral is equal to 
\begin{align}
\notag
\int \ud z_3
(z_1-z_2)&(z_1-z_3)(z_2-z_3)
\exp\left(\frac{\nu t}{2}z_3^2+(x_3+s_2+s_3)z_3\right)\\ \notag
=&
(z_1-z_2)\int \ud z_3
(z_1-z_3)(z_2-z_3)
\exp\left(\frac{\nu t}{2}z_3^2+(x_3+s_2+s_3)z_3\right)\\ \notag
=&(z_1-z_2)\exp\left(-\frac{\left(x_3+s_2+s_3\right)^2}{2 \nu  t}\right)\\ \notag
&\times \int \ud z_3
\: \exp\left(\frac{\nu t}{2}\left(z_3+\frac{x_3+s_2+s_3}{\nu t}\right)^2\right)
\\ \notag
&\quad\times 
\left[z_3+\frac{x_3+s_2+s_3}{\nu t} -\left(\frac{x_3+s_2+s_3}{\nu t}+z_1\right)\right]\\ \notag
&\quad\times \left[z_3+\frac{x_3+s_2+s_3}{\nu t} -\left(\frac{x_3+s_2+s_3}{\nu t}+z_2\right)\right]
\\ \notag
=&\imath (\nu t)^{-5/2}\sqrt{2 \pi } 
   \exp\left(-\frac{\left(x_3+s_2+s_3\right)^2}{2 \nu  t}\right)\\ \notag
   &\times
\left(z_1-z_2\right)\left(-\nu t+(\nu t z_1+x_3+s_2+s_3)(\nu t z_2+x_3+s_2+s_3)\right)\\
  =:& F_2(z_2).
  \label{E:F2}
\end{align}
The $\ud z_2$-integral can be obtained in the same way as above
\begin{align}
 \notag
\int \ud z_2& \exp\left(\frac{\nu t}{2}z_2^2+(x_2+s_1-s_3)z_2\right) F_2(z_2)\\ \notag
=&
\imath (\nu t)^{-5/2}\sqrt{2\pi}\:\exp\left(
-\frac{(x_3+s_2+s_3)^2}{2\nu t}
-\frac{(x_2+s_1-s_3)^2}{2\nu t}\right)\\  \notag
&\times \int \ud z_2 \exp\left(\frac{\nu t}{2}\left(z_2+\frac{x_2+s_1-s_3}{\nu t}\right)^2\right)
\\ \notag
&\quad\times
\left[z_1+\frac{x_2+s_1-s_3}{\nu t} - \left(z_2+\frac{x_2+s_1-s_3}{\nu t}\right)\right]\\ \notag
&\quad\times 
\left[-\nu t + (\nu tz_1 +x_3+s_2+s_3)\left(\nu t \left[z_2+\frac{x_2+s_1-s_3}{\nu t}\right]-s_1+s_2+2s_3-x_2+x_3\right)\right]
\\ \notag
=& \imath^2 (\nu t)^{-4} 2\pi (-s_1+s_2+2s_3-x_2+x_3) 
\exp\left(
-\frac{(x_3+s_2+s_3)^2}{2\nu t}
-\frac{(x_2+s_1-s_3)^2}{2\nu t}\right) 
\\ \notag
&\times
\left[\nu t + \left(\nu t z_1+x_3+s_2+s_3\right)
\left(\nu t z_1+x_2+s_1-s_3\right)\right]
\\=:& F_1(z_1).
\label{E:F1}
\end{align}
Similarly, one can calculate the $\ud z_1$-integral, which is equal to 
\begin{align}\notag
\int\ud z_1&
\exp\left(\frac{\nu t}{2}z_1^2+(x_1-s_1-s_2)z_1\right)F_1(z_1)\\ \notag
=&
\imath^2 (\nu t)^{-4}2\pi (-s_1+s_2+2s_3-x_2+x_3)\\ \notag
&\times \exp\left(
-\frac{(x_3+s_2+s_3)^2}{2\nu t}
-\frac{(x_2+s_1-s_3)^2}{2\nu t}
-\frac{(x_1-s_1-s_2)^2}{2\nu t}\right) \\ \notag
&\times \int\ud z_1
\exp\left(\frac{\nu t}{2}\left(z_1^2+\frac{x_1-s_1-s_2}{\nu t}\right)^2\right)\\ \notag
&\qquad\times\Bigg\{\nu t+\left[\nu t\left(z_1+\frac{x_1-s_1-s_2}{\nu t}\right)+s_1+2s_2+s_3-x_1+x_3\right]
\\ \notag
&\hspace{5em} \times \left[\nu t\left(z_1+\frac{x_1-s_1-s_2}{\nu t}\right)+2s_1+s_2-s_3-x_1+x_2\right]
\Bigg\}\\ \notag
=& \imath^3 (\nu t)^{-9/2} (2\pi)^{3/2}\\ \notag
&\times
\left(-s_1+s_2+2 s_3-x_2+x_3\right)
\left(s_1+2 s_2+s_3-x_1+x_3\right)
\left(2 s_1+s_2-s_3-x_1+x_2\right)
\\ \notag
&\times 
\exp\left(
-\frac{(x_3+s_2+s_3)^2}{2\nu t}
-\frac{(x_2+s_1-s_3)^2}{2\nu t}
-\frac{(x_1-s_1-s_2)^2}{2\nu t}\right) \\
=:&F_0(s_1,s_2,s_3). 
\label{E:F0}
\end{align}
Hence,
\begin{align}\notag
 &\E\left[u(t,x_1)u(t,x_2)u(t,x_3)\right]\\
 \notag
= & (\nu t)^{-9/2} (2\pi)^{-3/2}
\int_0^\infty \ud s_1 \: 
\exp\left(\frac{\lambda^2 s_1}{\nu}\right)
\int_0^\infty \ud s_2 \: 
\exp\left(\frac{\lambda^2 s_2}{\nu}\right)
\int_0^\infty \ud s_3 \: 
\exp\left(\frac{\lambda^2 s_3}{\nu}\right)\\
\notag
&\times
\left(-s_1+s_2+2 s_3-x_2+x_3\right)
\left(s_1+2 s_2+s_3-x_1+x_3\right)
\left(2 s_1+s_2-s_3-x_1+x_2\right) 
\\
&\times 
\exp\left(
-\frac{(x_3+s_2+s_3)^2}{2\nu t}
-\frac{(x_2+s_1-s_3)^2}{2\nu t}
-\frac{(x_1-s_1-s_2)^2}{2\nu t}\right).
\label{E_:3point}
\end{align}
Finally, for general $x_i\in\R$, $i=1,2,3$, without ordering, one obtains \eqref{E:3point} by symmetry.

{\bigskip\bf\noindent (2)~} In this part, we assume $x=x_1=x_2=x_3$. Then $F_0$ defined in \eqref{E:F0} is equal to
\begin{align}
\notag
F_0(s_1,s_2,s_3)=&\imath^3 (\nu t)^{-9/2} (2\pi)^{3/2} \\
\notag
&\times \left(-s_1+s_2+2 s_3\right)
\left(s_1+2 s_2+s_3\right)
\left(2 s_1+s_2-s_3\right) 
   \\
   &\times \exp \left(-\frac{s_1^2+s_2^2+s_3^2+s_1s_2-s_1s_3+s_2s_3}{\nu  t}
   -\frac{3x^2}{2\nu t}\right). 
\end{align}
Now we are going to apply Lemma \ref{L:RealInt} to evaluate the $\ud s_3$-integral.
First notice that
\begin{align}
\notag
&\left(-s_1+s_2+2 s_3\right)
\left(s_1+2 s_2+s_3\right)
\left(2 s_1+s_2-s_3\right) 
\\
\notag
&\quad=\left(-s_1+s_2\right) \left(2 s_1+s_2\right) \left(s_1+2s_2\right)
+3 \left(s_1^2+4 s_2 s_1+s_2^2\right)s_3
+3(s_1-s_2)s_3^2
-2s_3^3\\
&\quad=:\: a_0 + a_1s_3 + a_2 s_3^2 +a_3 s_3^3.
\end{align}
Hence, 
\begin{align}
 \notag
\int_0^\infty \ud s_3 \: \exp\left(\frac{\lambda^2}{\nu}s_3\right) F_0(s_1,s_2,s_3)
=&\imath^3 (\nu t)^{-9/2} (2\pi)^{3/2} \exp\left(-\frac{3x^2}{2 \nu t}-\frac{s_1^2+s_2^2+s_1s_2}{\nu t}\right)\\
\notag
&\times
\sum_{k=0}^3 a_k \int_0^\infty s_3^{k} \:\exp\left(-\frac{s_3^2}{\nu t}+\left(\frac{\lambda^2}{\nu}+\frac{s_1-s_2}{\nu t}\right) s_3\right)\ud s_3\\
\notag
=&\imath^3(\nu t)^{-9/2} (2\pi)^{3/2}
\exp\left(-\frac{3x^2}{2 \nu t}-\frac{s_1^2+s_2^2+s_1s_2}{\nu t}\right)\\
&\times
\sum_{k=0}^3 a_k \: \Lambda_k\left(\frac{\lambda^2}{\nu}+\frac{s_1-s_2}{\nu t},\nu t\right),
\end{align}
where the functions $\Lambda_k(\cdot,\cdot)$ are defined in \eqref{E:DefLambda}.
After some (tedious) expansion and simplification, 
we see that
\begin{align}
\notag
\sum_{k=0}^3 a_k& \Lambda_k\left(\frac{\lambda^2}{\nu} + \frac{s_1-s_2}{\nu t},\nu t\right)\\ 
\notag
 =&\quad  b_0 + b_1 s_2 + b_2 s_2^2 \\
 \notag
 &-\frac{\sqrt{\pi\nu} \lambda^2 t^{3/2}}{4} 
   \exp\left(\frac{\left(s_1-s_2+\lambda^2 t\right)^2}{4 \nu t}\right) \left(t \left(6 \nu +\lambda^4 t\right)-9
   \left(s_1+s_2\right)^2\right)\\
\notag   
 &+\frac{\sqrt{\pi\nu} \lambda^2 t^{3/2}}{8} 
   \left(t \left(6 \nu +\lambda^4 t\right)-9
   \left(s_1+s_2\right)^2\right)
   \exp\left(\frac{\left(s_1-s_2+\lambda^2 t\right)^2}{4 \nu t}\right)\Erfc\left(\frac{s_1-s_2+\lambda^2 t}{2
   \sqrt{\nu t}}\right)\\
   =:& A_1(s_2) + A_2(s_2) +A_3(s_2),
\end{align}
where
\begin{align}
b_0 &= \nu t \left(2s_1^2+\frac{\lambda^2 t}{4} s_1 -\nu t-\frac{\lambda^4 t^2}{4}\right),\qquad
b_1 = \nu t \left(5s_1-\frac{\lambda^2 t}{4}\right)
\quad\text{and}\quad
b_2 =2 \nu t.
\end{align}
Therefore, the $\ud s_2$-integral is equal to
\begin{multline}
\int_0^\infty \ud s_2 \:\exp\left(\frac{\lambda^2}{\nu}s_2\right)
\int_0^\infty \ud s_3 \: \exp\left(\frac{\lambda^2}{\nu}s_3\right) F_0(s_1,s_2,s_3) 
=\imath^3 (\nu t)^{-9/2} (2\pi)^{3/2} \\
\times \exp\left(-\frac{3x^2}{2 \nu t}-\frac{s_1^2}{\nu t}\right)
\sum_{k=0}^3\int_0^\infty\ud s_2\: \exp\left(-\frac{s_2^2}{\nu t} + \left(\frac{\lambda^2}{\nu}-\frac{s_1}{\nu t}\right) s_2 \right)
  A_k(s_2).
\end{multline}
By Lemma \ref{L:RealInt},
\begin{align}
 \notag
\int_0^\infty\ud s_2\: &\exp\left(-\frac{s_2^2}{\nu t} + \left(\frac{\lambda^2}{\nu}-\frac{s_1}{\nu t}\right) s_2 \right) A_1(s_2)\\
\notag
=&
\sum_{k=0}^2 b_k \Lambda_k\left(\frac{\lambda^2}{\nu}-\frac{s_1}{\nu t},\nu t\right)\\
\notag
=&\quad \frac{1}{8} \nu ^2 t^2 \left(16 s_1+3 \lambda^2
   t\right)\\
 &+\frac{1}{16} \sqrt{\pi} \lambda^2\nu^{3/2}t^{5/2}
   \left(15 s_1+\lambda^2 t\right)
   \exp\left(\frac{\left(s_1-\lambda^2 t\right)^2}{4 \nu 
   t}\right) \left(2-\Erfc\left(\frac{\lambda^2 t-s_1}{2
   \sqrt{\nu t}}\right)\right).
\end{align}
Similarly, for the $A_2$ term,  we have that
\begin{align}\notag
\int_0^\infty\ud s_2\:& \exp\left(-\frac{s_2^2}{\nu t} + \left(\frac{\lambda^2}{\nu}-\frac{s_1}{\nu t}\right) s_2 \right) A_2(s_2) \\
\notag
=&\frac{1}{4} \sqrt{\pi\nu} \lambda^2 t^{3/2}
\exp\left(\frac{\left(s_1+\lambda^2 t\right)^2}{4 \nu t}\right)
\\
\notag
&\times \int_0^\infty\ud s_2\: \exp\left(-\frac{3s_2^2}{4\nu t} + \left(\frac{\lambda^2}{2\nu}-\frac{3s_1}{2\nu t}\right) s_2 \right)
   \left(9\left(s_1+s_2\right)^2-t \left(6 \nu +\lambda^4 t\right)\right)
\\
=&\frac{1}{2} \sqrt{\pi } \lambda^2 \nu^{3/2}  t^{5/2}
   \left(3 s_1+\lambda^2 t\right)
   \exp\left(\frac{\left(s_1+\lambda^2 t\right)^2}{4 \nu t}\right),
\end{align}
where we have applied Lemma \ref{L:RealInt} in the last step.
The integration with respect to $A_3$ is much more complicated. Instead, we claim that for $r>0$,
\begin{align}
\notag
I(r):=&\int_0^r\ud s_2\: \exp\left(-\frac{s_2^2}{\nu t} + \left(\frac{\lambda^2}{\nu}-\frac{s_1}{\nu t}\right) s_2 \right) A_3(s_2) \\ \notag
=&\frac{1}{8} \sqrt{\pi\nu} \lambda^2 t^{3/2}
\int_0^r\ud s_2\: \exp\left(-\frac{s_2^2}{\nu t} + \left(\frac{\lambda^2}{\nu}-\frac{s_1}{\nu t}\right) s_2 \right)
   \left(t \left(6 \nu +\lambda^4 t\right)-9\left(s_1+s_2\right)^2\right)
\\ \notag
&\times \exp\left(\frac{\left(s_1-s_2+\lambda^2 t\right)^2}{4 \nu t}\right)\Erfc\left(\frac{s_1-s_2+\lambda^2 t}{2
   \sqrt{\nu t}}\right)\\ \notag
=& \frac{1}{16} \lambda^2 \nu  t^2 
   \Bigg(\sqrt{\pi\nu t}
   \Bigg[\lambda^2 t  \Bigg\{-4 \exp\left(\frac{(s_1+\lambda^2 t)^2}{4\nu t}\right) \Erfc\left(\frac{s_1+\lambda^2
   t}{2 \sqrt{\nu  t}}\right)\\ \notag
   & +5 
   \exp\left(\frac{(s_1-\lambda^2 t)^2}{4\nu t}\right)
   \left(\Erfc\left(\frac{2 r+s_1-\lambda^2t}{2 \sqrt{\nu 
   t}}\right)+\Erfc\left(\frac{\lambda^2 t-s_1}{2
   \sqrt{\nu  t}}\right)\right)\Bigg\}\\ \notag
   &+4 \left(3 r+\lambda
   ^2 t\right) \exp\left(-\frac{3 r^2}{4 \nu  t}+\frac{r \left(2 \lambda^2 t-6
   s_1\right)}{4 \nu  t}+\frac{2 \lambda^2 s_1
   t+s_1^2+\lambda^4 t^2}{4 \nu  t}\right)
   \Erfc\left(\frac{-r+s_1+\lambda^2 t}{2
   \sqrt{\nu  t}}\right)\Bigg]\\ \notag
   &+3 \sqrt{\pi\nu  t} s_1
  \Bigg[
  -4 \exp\left(\frac{(s_1+\lambda^2 t)^2}{4\nu t }\right)
   \Erfc\left(\frac{s_1+\lambda^2 t}{2 \sqrt{\nu t}}\right) \\ \notag
   &+\exp\left(\frac{(s_1-\lambda^2 t)^2}{4 \nu t}\right)
   \Erfc\left(\frac{\lambda^2 t-s_1}{2 \sqrt{\nu t}}\right)\\ \notag
   &+\exp\left(\frac{(s_1-\lambda^2 t)^2}{4 \nu t}\right) 
   \Erfc\left(\frac{2 r+s_1-\lambda^2 t}{2
   \sqrt{\nu  t}}\right)\\ \notag
   &+4 \exp\left(-\frac{3 r^2}{4 \nu  t}+\frac{r \left(2 \lambda^2 t-6
   s_1\right)}{4 \nu  t}+\frac{2 \lambda^2 s_1
   t+s_1^2+\lambda^4 t^2}{4 \nu  t}\right)
   \Erfc\left(\frac{-r+s_1+\lambda^2 t}{2
   \sqrt{\nu  t}}\right)\\ \notag
   &-2 \exp\left(\frac{(s_1-\lambda^2 t)^2}{4 \nu  t}\right)\Bigg]\\ 
   &-2 t \Bigg[5 \sqrt{\pi\nu t}
   \: \lambda^2 
   \exp\left(\frac{(s_1-\lambda^2 t)^2}{4\nu t}\right)
   +3 \nu 
   -3 \nu  \exp\left(-\frac{r^2}{\nu  t}-\frac{r \left(s_1-\lambda^2
   t\right)}{\nu  t}\right)\Bigg]\Bigg),
\end{align}
which can be verified directly by differentiating both sides.
Because $\Erfc(x)\sim x^{-1}\pi^{-1/2} e^{-x^2}$ as $x\rightarrow\infty$ and $\lim_{x\rightarrow-\infty}\Erfc(x)=2$ , we see that 
\begin{align}\notag
\lim_{r\rightarrow\infty}I(r) 
=& \frac{1}{16} \lambda^2 \nu  t^2 
   \Bigg(\sqrt{\pi\nu t}\:
   \Bigg[\lambda^2 t \:  \Bigg\{-4 \exp\left(\frac{(s_1+\lambda^2 t)^2}{4\nu t}\right) \Erfc\left(\frac{s_1+\lambda^2
   t}{2 \sqrt{\nu  t}}\right)\\ \notag
   &+5\:
   \exp\left(\frac{(s_1-\lambda^2 t)^2}{4\nu t}\right)\Erfc\left(\frac{\lambda^2 t-s_1}{2
   \sqrt{\nu  t}}\right)\Bigg\}+0\Bigg]\\ \notag
   &+3 \sqrt{\pi\nu  t} s_1
   \Bigg[-4 \exp\left(\frac{(s_1+\lambda^2 t)^2}{4\nu t}\right)
   \Erfc\left(\frac{s_1+\lambda^2 t}{2 \sqrt{\nu t}}\right) \\ \notag
   &+\exp\left(\frac{(s_1-\lambda^2 t)^2}{4 \nu  t}\right)
   \Erfc\left(\frac{\lambda^2 t-s_1}{2 \sqrt{\nu t}}\right)\\ \notag
   &+0+0-2 \exp\left(\frac{(s_1-\lambda^2 t)^2}{4 \nu  t}\right)\Bigg]\\ 
   &-2 t \Bigg[5 \sqrt{\pi\nu t}\:
   \lambda^2 \exp\left(\frac{(s_1-\lambda^2 t)^2}{4\nu t}\right)+3 \nu 
   -0\Bigg]\Bigg).
\end{align}
Therefore,
\begin{align}\notag
\int_0^\infty\ud s_2\:& \exp\left(-\frac{s_2^2}{\nu t} + \left(\frac{\lambda^2}{\nu}-\frac{s_1}{\nu t}\right) s_2 \right) A_3(s_2)\\ \notag
=&   \frac{1}{16} \lambda^2 \nu  t^2 
   \Bigg(\sqrt{\pi\nu}\:
   \lambda^2 t^{3/2} \:  \Bigg\{-4 \exp\left(\frac{(s_1+\lambda^2 t)^2}{4\nu t}\right) \Erfc\left(\frac{s_1+\lambda^2
   t}{2 \sqrt{\nu  t}}\right)\\ \notag
   &+5\:
   \exp\left(\frac{(s_1-\lambda^2 t)^2}{4\nu t}\right)\Erfc\left(\frac{\lambda^2 t-s_1}{2
   \sqrt{\nu  t}}\right)\Bigg\}\\ \notag
   &+3 \sqrt{\pi\nu  t} s_1
   \Bigg[-4\: \exp\left(\frac{(s_1+\lambda^2 t)^2}{4\nu t}\right)\Erfc\left(\frac{s_1+\lambda^2 t}{2 \sqrt{\nu t}}\right)\\ \notag
   &+\exp\left(\frac{(s_1-\lambda^2 t)^2}{4 \nu  t}\right)
   \Erfc\left(\frac{\lambda^2 t-s_1}{2 \sqrt{\nu t}}\right)-2 \exp\left(\frac{(s_1-\lambda^2 t)^2}{4 \nu  t}\right)\Bigg]\\
   &-2 t \Bigg[5 \sqrt{\pi \nu t}\:
   \lambda^2  \exp\left(\frac{(s_1-\lambda^2 t)^2}{4\nu t}\right)+3 \nu \Bigg]\Bigg).
\end{align}
Finally, combining all these three integrals, we see that
\begin{align}\notag
\exp\left(\frac{\lambda^2}{\nu} s_1\right)&\int_0^\infty \ud s_2 \: \exp\left(\frac{\lambda^2}{\nu}s_2\right)
\int_0^\infty \ud s_3 \: \exp\left(\frac{\lambda^2}{\nu}s_3\right) F_0(s_1,s_2,s_3) \\ \notag
 =&  \imath^3 \exp\left(-\frac{3x^2}{2\nu t}\right)\\ \notag
 & \times \Bigg[\frac{\pi^2 \lambda^2  }{\sqrt{2}\:\nu^3t^2}
 \left(3
   s_1+\lambda^2 t\right) \exp\left(\frac{\lambda^4t}{\nu}-\frac{3\left(s_1-\lambda^2
   t\right)^2}{4 \nu  t}\right)
   \left(2-\Erfc\left(\frac{s_1+\lambda^2 t}{2
   \sqrt{\nu  t}}\right)\right)
   \\ \notag
 &\quad+\frac{\pi^2 \lambda^2 
 \left(3 s_1-\lambda^2 t\right)}{\sqrt{2}\: \nu^3 t^2 }
      \exp\left(
      \frac{t\lambda^4}{3\nu}-\frac{3\left(s_1-\lambda^2 t/3\right)^2}{4 \nu t}\right) 
      \left(2-\Erfc\left(\frac{\lambda^2 t-s_1}{2\sqrt{\nu  t}}\right)\right)\\ \notag
   &\quad+ 2^{5/2}(\nu t)^{-5/2} \pi^{3/2} s_1
   \exp\left(\frac{\lambda^4t}{4\nu}-\frac{\left(s_1-\lambda^2 t/2\right)^2}{\nu t}\right) \Bigg]\\
   =:& \Theta(s_1).
\end{align}
Therefore, the third moment is equal to 
\begin{align}
\E\left[u(t,x)^3\right]=
\frac{1}{(2\pi\imath)^3}
\int_0^\infty  
\Theta(s_1)\ud s_1.
\end{align}
Because $2-\Erfc(x)=2\Phi\left(\sqrt{2} x\right)$ and thanks to Lemma \ref{L:asbcsd},
\begin{align}\notag
\frac{1}{(2\pi)^3}& \exp\left(-\frac{3x^2}{2\nu t}\right)\int_0^\infty 
2^{5/2}(\nu t)^{-5/2} \pi^{3/2} s_1
   \exp\left(
   \frac{\lambda^4t}{4\nu}
   -\frac{\left(s_1-\lambda^2 t/2\right)^2}{\nu t}\right)
   \ud s_1\\ \notag
=&\frac{1}{\sqrt{2} \pi^{3/2}
   (\nu t)^{5/2}}\exp\left(-\frac{3x^2}{2\nu t}\right)
\int_0^\infty s_1 \exp\left(-\frac{s_1 \left(s_1-
   \lambda^2 t\right)}{\nu t}\right)\ud s_1\\
=&\frac{1}{(2\pi\nu t)^{3/2}}
\exp\left(-\frac{3x^2}{2\nu t}\right)
   +\frac{\lambda^2}{2\sqrt{2} \pi\nu^2 t}
   \exp\left(\frac{\lambda^4 t}{4\nu}-\frac{3x^2}{2\nu t}\right) \left(1-\Phi\left(-\lambda^2 \sqrt{\frac{t}{2\nu}}\right)\right),
\end{align}
one can obtain \eqref{E:3moment} after some simplification.

{\bigskip\bf\noindent (3)~}
By the mean-value theorem, there exit two constants $a,b$ in the range given by \eqref{E:range-ab} such that
\begin{align}
\notag
\E[u(t,x)^3] 
=& \quad \frac{b \lambda^2  }{2^{5/2}\pi\:\nu^3t^2}\exp\left(\frac{\lambda^2 t}{4\nu}-\frac{3x^2}{2\nu t}\right)
\int_0^\infty  
 \left(3
   s_1+\lambda^2 t\right) 
   \exp\left(
   -\frac{3s_1(s_1-2\lambda^2t)}{4 \nu  t}\right)
 \ud s_1  
   \\
   \notag
 &+ 
 \frac{a \lambda^2 }{2^{5/2}\pi\:\nu^3t^2}\exp\left(\frac{\lambda^2 t}{4\nu}-\frac{3x^2}{2\nu t}\right)
  \int_0^\infty \left(3 s_1-\lambda^2 t\right)
      \exp\left(-\frac{s_1(3s_1-2\lambda^2 t)}{4 \nu t}\right) 
      \ud s_1\\ \notag
   &+\frac{1}{(2\pi\nu t)^{3/2}}
\exp\left(-\frac{3x^2}{2\nu t}\right)\\
&
   +\frac{\lambda^2}{2\sqrt{2} \pi\nu^2 t}
   \exp\left(\frac{\lambda^4 t}{4\nu}-\frac{3x^2}{2\nu t}\right) \left(1-\Phi\left(-\lambda^2 \sqrt{\frac{t}{2\nu}}\right)\right).
\end{align}
Then one can apply Lemma \ref{L:asbcsd} to evaluate these two integrals. 
This completes the whole proof of Theorem \ref{T:3rdMoment}.
\end{proof}

\section*{Acknowledgements}
L.C. would like to thank {\em Robert Dalang} and {\em Davar Khoshnevisan} for many helpful comments.
Khoshnevisan asked L.C. in 2014 whether one could obtain an explicit formula for the third moment, which motivated the current study. There was no much progress on this problem until when L.C. was vising the {\it Simons Center for Geometry and Physics} for the conference 
-- {\it Stochastic Partial Differential Equations} (May 16-20, 2016), {\em Ivan Corwin} pointed out to L.C. the transform from \eqref{E:BCMac} to \eqref{E:BCMac2}. This became the starting point of the whole calculation in this paper. Here L.C. would like to express his sincere gratitude to him.
Finally, L.C. would also like to thank the organizer {\em Martin Hairer} for the wonderful conference.

\begin{small}
\bigskip


\end{small}

\end{document}